\documentclass[review]{elsarticle}
\usepackage{lineno}
\usepackage[latin1]{inputenc}
\usepackage{times}
\usepackage{amssymb,mathrsfs, amsmath}
\usepackage{amsmath,amsthm}
\usepackage{amssymb}
\usepackage{latexsym}
\usepackage{amsfonts}
\usepackage{epsfig}
\usepackage{hyperref}
\usepackage{graphicx}
\usepackage{graphics}
\usepackage[all]{xy}
\usepackage{tikz-cd}
\usepackage[mathscr]{euscript}
\usepackage{mathrsfs}
\usepackage{stmaryrd}

\newtheorem{thm}{Theorem}[section]

\newtheorem{defn}{Definition}[section]
\newtheorem{prop}{Proposition}[section]

\newtheorem{lem}{Lemma}[section]

\newtheorem{rem}{Remark}[section]

\newtheorem{cor}{Corollary}[section]

\newtheorem{exmpl}{Example}[section]

\newtheorem{obsn}{Observation}[section]
\journal{... }

\DeclareMathOperator{\Hom}{Hom}

\begin{document}
\begin{frontmatter}
\title{Crossed  Homomorphisms on  associative superalgebras}

%\author{\fnref{myfootnote}}
%\address{}
%\fntext[myfootnote]{This author is supported by CSIR,\textsc{India}.}
%\author{RB Yadav \fnref{myfootnote}\corref{mycorrespondingauthor}}
%\address{Indian Statistical Institute Tezpur, Assam 784028, \textsc{India}}
%\ead{rbyadav15@gmail.com}
%\author{Subir Mukhopadhyay\fnref{myfootnote}\corref{mycorrespondingauthor}}
%\address{Indian Statistical Institute Tezpur, Assam 784028, \textsc{India}, \textsc{India}}
%\ead{goutam@isical.ac.in}
%\fntext[myfootnote]{This author is supported by CSIR,\textsc{India}.}

%
\author{RB Yadav \fnref{myfootnote}\corref{mycorrespondingauthor}}
%\address{Sikkim University, Gangtok, Sikkim, 737102, \textsc{India}}
\cortext[mycorrespondingauthor]{Corresponding author}
\ead{rbyadav01@cus.ac.in, rbyadav15@gmail.com}
\author{Arpan Sharma\corref{mycoauthor}}
%%\address{Sikkim University, Gangtok, Sikkim, 737102, \textsc{India}}
\ead{arpansachin@gmail.com}
%\author{ Rinkila Bhutia\corref{mycoauthor}}
\address{Sikkim University, Gangtok, Sikkim, 737102, \textsc{India}}
%\ead{rbhutia@cus.ac.in}
\begin{abstract}
 In this paper, we introduce the notion of crossed homomorphisms on associative superalgebras. We show that crossed homomorphisms are precisely the Maurer--Cartan elements of a naturally associated graded Lie algebra. This characterization enables us to construct a cohomology theory of crossed homomorphisms.  We then investigate formal deformations of crossed homomorphisms and derive the corresponding deformation equations. It is proved that the infinitesimal part of a formal deformation is a $1$-cocycle in the associated cohomology, while the obstruction to extending a deformation of finite order is  a $2$-cocycle.
\end{abstract}

\begin{keyword}
\texttt{  Cohomology, Crossed homomorphism, Deformation, Maurer-Cartan element, Associative superalgebra, Graded Lie Algebra}
\MSC[2020]  	16E40, 16S80, 	,  17B56, 17B70,  17B10, 17B55
\end{keyword}
\end{frontmatter}
%\linenumbers
%\maketitle

%\maketitle

\section{Introduction}
Crossed homomorphisms (also called relative difference operators) arise naturally whenever a group, Lie algebra, associative algebra, or superalgebra acts on another algebraic object of the same kind. They  appear in many branches of pure mathematics and also in several areas of theoretical computer science and mathematical physics. The concept of crossed homomorphisms for  Lie algebras is in \cite{Lue1966}. In recent years, crossed homomorphisms have been studied by several scholars \cite{MR4499354}, \cite{Apurba}, \cite{Wang},  \cite{MR4708529}, \cite{MR4957732}, \cite{MR4908221}, \cite{RB-Arpan}.
Associative superalgebras play a fundamental role in graded algebra, supergeometry, and mathematical physics. Cohomology and deformation theory of associative algebra was studied in \cite{MR4611401}. Since many classical algebraic constructions admit meaningful $\mathbb{Z}_2$-graded analogues, it is natural to investigate crossed homomorphisms in the setting of associative superalgebras and to develop the corresponding cohomological and deformation theories.
 Given a graded Lie algebra, its Maurer--Cartan elements often describe algebraic structures of interest, while the associated twisted differential governs their infinitesimal deformations. This viewpoint has become one of the central tools in modern deformation theory, providing a unified framework for studying deformations and their obstructions.
The primary objective of this paper is to develop such a framework for crossed homomorphisms on associative superalgebras.
We first introduce the notion of a crossed homomorphism on an associative superalgebra with coefficients in a bimodule. We then construct a graded Lie algebra on the space of multilinear maps and prove that a linear map is a crossed homomorphism if and only if it satisfies the Maurer--Cartan equation in this graded Lie algebra. This characterization identifies the defining equation of crossed homomorphisms with a Maurer--Cartan equation and provides the algebraic setting required for cohomological and deformation-theoretic investigations.
Using the Maurer--Cartan characterization, we define the cohomology of a crossed homomorphism.  The associated cohomology groups naturally arise in their deformation theory.
Finally, we study formal deformations of crossed homomorphisms. We derive the deformation equations  and show that the first-order term of every formal deformation is a $1$-cocycle in the associated cohomology. Furthermore, we prove that the obstruction to extending a deformation of order $n$ to order $n+1$ is a $2$-cocycle. Consequently, the cohomology introduced in this paper governs both infinitesimal deformations and obstruction theory for crossed homomorphisms on associative superalgebras.
The paper is organized as follows. Section $1$ recalls the necessary preliminaries on associative superalgebras and graded Lie algebras.  In Section $2$, we introduce definition of crossed homomorphism on associative superalgebra and give some interesting examples.  In Section $3$, we introduce crossed homomorphisms and characterize them as Maurer--Cartan elements. In Section $4$ develops the associated cohomology theory. In Section $5$, we investigate formal deformations and establish the corresponding obstruction theory.
\section{Crossed homomorphisms between associative superalgebras}
In this section, we study crossed homomorphisms between associative superalgebras and relate this notion to crossed homomorphisms of Lie superalgebras. Let $\Delta$ be a commutative group. A $\Delta$-graded algebra   over $K$ is a $\Delta$-graded vector space $E=\bigoplus_{\alpha\in\Delta}E^\alpha$ together with a bilinear map $m:E\times E\rightarrow E$ such that $m(E^\alpha\times  E^\beta)\subset E^{\alpha+\beta}$ for all $\alpha,\;\beta\in\Delta$. 
\begin{defn}
  An  associative superalgebra  is a $\mathbb{Z}_2$-graded algebra $A=A_{0}\oplus A_{1}$ such that $m(m(a,b),c)=m(a,m(b,c)),$ for all $a,b,c\in A.$   
\end{defn}
\begin{rem}
    Every associative superalgebra is an associative algebra but converse is not true. An associative algebra structure $m$ on a $\mathbb{Z}_2$-graded vector space  $A=A_{0}\oplus A_{1}$ is associative superalgebra structure if $m$ is homogeneous of even degree. 
\end{rem}
\begin{exmpl} Let $A=A_0\oplus A_1,$  where $A_0=\mathrm{span}\{e\},
\quad
A_1=\mathrm{span}\{f\}.$ Define a multiplication on $A$ by $e^2=e,\quad ef=f,\quad fe=f,\quad f^2=e+f,$ 
and extend it bilinearly. Then $A$ is an associative algebra. Indeed, one checks associativity on basis elements. For example, $ 
(f^2)f=(e+f)f=f+(e+f)=e+2f,$ 
while $  f(f^2)=f(e+f)=f+(e+f)=e+2f.$ 
Similarly, all other products satisfy associativity.
However, $A$ is not an associative superalgebra with respect to the given grading. Indeed, $ 
f\in A_1,$ 
so in a superalgebra one must have $f^2\in A_0.$ 
But $  f^2=e+f\notin A_0,$ 
since it has a nonzero component in $A_1$.
Hence $A=A_0\oplus A_1$ is an associative algebra but not an associative superalgebra.
\end{exmpl}
\begin{exmpl}\label{rbLSe1}
 Let $V=V_{0}\oplus V_{1}$ be a $\mathbb{Z}_2$-graded vector  space, $dim V_{0} = m$, $dim V_{1} = n$. Consider the  set  $End V$ of all endomorphisms of $V$. 
Define
\begin{equation}
 End_i \; V = \{a \in End \;V\;|\; a V_s\subseteq V_{i+s}\} \; ,\; i,s\in \mathbb{Z}_2
 \end{equation}
 One can easily verify that  $End\; V=End_{0}\; V\oplus End_{1} \; V  $. $End\;V$ is an associative superalgebra with respect to composition operation.
In some (homogeneous) basis of $V$, $End\; V$ ( also denoted by $\ell(m,n)$ ) consists of  block matrices of the form $\big(\begin{smallmatrix}\alpha & \beta \\ \gamma & \delta \end{smallmatrix}\big)$, where $\alpha, \beta, \gamma,  \delta$ are matrices of order $m\times m$, $m\times n$, $n\times m$ and $n\times n,$ respectively.
\end{exmpl}

 Let $A$ and $B$ be associative superalgebras. Suppose that $A$ acts on $B$, i.e., $B$ is an $A$-bimodule. This means there exist degree $0$  bilinear maps
\[
l : A \otimes B \to B, \quad (a,x) \mapsto a \cdot x,
\qquad
r : B \otimes A \to B, \quad (x,a) \mapsto x \cdot a,
\]
such that for all homogeneous elements $a,b \in A$ and $x,y \in B$, 
 $(ab)\cdot x = a \cdot (b \cdot x)$,
 $(a \cdot x)\cdot b = a \cdot (x \cdot b)$,
 $(x \cdot a)\cdot b = x \cdot (ab)$,
 $(a \cdot x)\cdot y = a \cdot (xy),$
 $(x \cdot a)\cdot y = x \cdot (a \cdot y),$  
 $(xy)\cdot a = x \cdot (y \cdot a).$ 
In this case, we call $B$ as  an associative $A$-bimodule. In particular, any associative superalgebra $A$ is an associative $A$-bimodule via the adjoint action.

\begin{defn}
Let $B$ be   an associative $A$-bimodule.  We call a  linear map $H : A \to B$ of degree $0$  a \emph{crossed homomorphism of associative superalgebras} if for all homogeneous $a,b \in A$,
\begin{equation}\label{SCH1}
  H(ab) = a \cdot H(b) + H(a) \cdot b + H(a) H(b).  
\end{equation}
\end{defn}
\begin{obsn}
Every  crossed homomorphism of associative superalgebras is also a crossed homomorphism of corresponding associative algebras but converse is not true.   
\end{obsn}
\begin{exmpl} Consider the associative algebra $A=A_0\oplus A_1,$  where
$A_0=\mathrm{span}\{e\}$ and  $ A_1=\mathrm{span}\{f\}$,  
with multiplication $e^2=e,\quad ef=f,\quad fe=f,\quad f^2=e.$
Define a linear map $H:A\to A$  by  $H(e)=0,\quad H(f)=e-f.$
We show that $H$ is a crossed homomorphism of the associative algebra $A$, namely, $H(ab)=H(a)b+aH(b)+H(a)H(b)$  for all $a,b\in A$. It suffices to check this on basis elements. \\For $(e,e)$:
$$
H(e^2)=H(e)=0,\;
H(e)e+eH(e)+H(e)H(e)=0.
$$
For $(e,f)$:
$$
H(ef)=H(f)=e-f,\;
H(e)f+eH(f)+H(e)H(f)
=0+e(e-f)+0=e-f.
$$
For $(f,e)$:
$$
H(fe)=H(f)=e-f,\;
H(f)e+fH(e)+H(f)H(e)
=(e-f)e+0+0=e-f.
$$
For $(f,f)$:
$$ H(f^2)=H(e)=0,\; H(f)f+fH(f)+H(f)H(f)
=(e-f)f+f(e-f)+(e-f)(e-f)
=0.$$  
%Over a field of characteristic $2$, we have $ 
%f+f=0,$  hence $ 
%f+f+e=e=H(f^2).$ 
Therefore $H$ is a crossed homomorphism of the associative algebra $A$. %(over a field of characteristic $2$).
However, $H$ is not a crossed homomorphism of an associative superalgebra, since a crossed homomorphism of associative superalgebras is required to be of  even degree, i.e. $ 
H(A_i)\subseteq A_i.$ 
But $f\in A_1,
\quad
H(f)=e-f \notin A_1,$  so $H$ does not preserve the grading.
Hence $H$ is a crossed homomorphism of the associative algebra structure on $A$, but not of the corresponding $\mathbb Z_2$-graded (super) structure.
\end{exmpl}
\begin{exmpl} Define $ H:End(V)\to End(V), \qquad H(a)=a. $ 
Then $H$ is an even linear map. For homogeneous $ a,b\in End(V),$
we have $ H(ab)=ab.$ On the other hand, $H(a)b+aH(b)+H(a)H(b)
=ab+ab+ab.$ 
Thus $H$ is not a crossed homomorphism over a field of characteristic different from $2$.
\end{exmpl}
\begin{exmpl} Let $A$ be an associative superalgebra. Let $  \phi:A\to A$ be an even associative superalgebra morphism, and define $ H(a)=\phi(a)-a.$  Then for homogeneous $a,b\in A$, we have \begin{align}
H(ab)&=\phi(ab)-ab \nonumber\\
&=\phi(a)\phi(b)-ab \nonumber\\
&=(H(a)+a)(H(b)+b)-ab \nonumber\\
&=H(a)H(b)+H(a)b+aH(b)\nonumber.
\end{align}
Therefore, $H(ab)=H(a)b+aH(b)+H(a)H(b),$ 
so $H$ is a crossed homomorphism.
\end{exmpl}
\begin{rem}
\begin{enumerate}
% \item A linear map $d : A \to A$ satisfying
% \[
% d(ab) = a d(b) + d(a) b + \lambda\, d(a)d(b)
% \]
% is called a differential operator of weight $\lambda$. Thus, a crossed homomorphism $H : A \to A$ is precisely a differential operator of weight $1$.

\item If the action of $A$ on $B$ is trivial, then a crossed homomorphism $H : A \to B$ reduces to a superalgebra morphism.

\item If the multiplication on $B$ is trivial, then a crossed homomorphism $H : A \to B$ is simply a derivation of $A$ with values in the $A$-bimodule $B$.
\end{enumerate}
\end{rem}

\begin{defn}
Let $H, H' : A \to B$ be two crossed homomorphisms from an associative superalgebra $A$ to an associative superalgebra $B$.  
We define a \emph{morphism} from $H$ to $H'$ as a pair $ (f_A,f_B)$ of two algebra homomorphisms  $f_A : A \to A, \qquad f_B : B \to B$ such  that 
$f_B \circ H = H' \circ f_A\;\; \text{and} \;  f_B(ax) = f_A(a) f_B(x),\; f_B(xa) = f_B(x) f_A(a)$, 
 for all $a \in A$ and $x \in B$.
 Let  $H, H', H'' : A \to B$ are crossed homomorphisms from an associative superalgebra $A$ to an associative superalgebra $B$. If $(f_A,f_B)$, $(f_A',f_B')$ are morphisms from $H$ to $H'$ and $H'$ to $H''$, respectively, then $(f_A'\circ f_A, f_B'\circ f_B)$ is a morphisms from $H$ to $H''$. $(I_A,I_B)$ is a morphism from $H$ to itself. This gives a category of crossed homomorphisms from $A$ to $B.$ We denote this category by $SCH(A,B)$.
\end{defn}

\begin{defn}
\begin{enumerate}
\item We define a pair $(A, H_A)$ consisting of an associative superalgebra $A$ and a crossed homomorphism $H_A : A \to A$. We call this pair as    \emph{Sup AssCH pair}.
\item Let $(A, H_A)$ and $(B, H_B)$ be two Sup AssCH pairs. A \emph{morphism} between them is a superalgebra morphism $f : A \to B$ such that $H_B \circ f = f \circ H_A.$ 
\end{enumerate}
This gives a category of  Sup AssCH pairs. We  denote the category of  Sup AssCH pairs by $\mathbf{Sup AssCH}$.
\end{defn}
\begin{lem}\label{TA1}
Let $H : A \to B$ be a crossed homomorphism. Then there is a new associative $A$-bimodule structure on $B$ defined by
\[
l_H(a,x) = ax + H(a)x, \qquad r_H(x,a) = xa + xH(a),
\]
for all $a \in A$ and $x \in B$.
\end{lem}

\begin{proof}
For $a,b \in A$ and $x \in B$, we compute:
$$
\begin{aligned}
l_H(a, l_H(b,x)) 
&= l_H(a, bx + H(b)x) \\
&= a(bx + H(b)x) + H(a)(bx + H(b)x) \\
&= (ab)x + \{aH(b) + H(a)b + H(a)H(b)\}x \\
&= (ab)x + H(ab)x \; \;\; \{\text{by  Equation }\;\ref{SCH1}\}\\
&= l_H(ab,x),
\end{aligned}
$$
Next,
$$ 
\begin{aligned}
r_H(l_H(a,x), b)
&= r_H(ax + H(a)x, b) \\
&= (ax + H(a)x)b + (ax + H(a)x)H(b) \\
&= a(xb) + H(a)(xb) + a(xH(b)) + H(a)(xH(b)) \\
&= l_H(a, xb + xH(b)) \\
&= l_H(a, r_H(x,b)).
\end{aligned}
$$ 
Similarly,
$$ 
\begin{aligned}
r_H(r_H(x,a), b)
&= r_H(xa + xH(a), b) \\
&= (xa + xH(a))b + (xa + xH(a))H(b) \\
&=x(ab) + x\{ H(a)b+aH(b)+H(a)H(b)\}\\
&= x(ab) + xH(ab) \; \;\; \{\text{by  Equation }\;\ref{SCH1}\}\\
&= r_H(x, ab).
\end{aligned}
$$ 
Finally, one checks easily that
$$
l_H(a,x)y = l_H(a,xy), \qquad r_H(xy,a) = x r_H(y,a),\qquad
r_H(x,a)y=xl_H(a,y).$$
So $B$ is an associative $A$-bimodule with respect to actions $l_H$ and $r_H$.
\end{proof}

Since $B$ is an associative $A$-bimodule, the direct sum $A \oplus B$ admits the structure of a semidirect product superalgebra given by $  (a,x)\cdot (b,y) = \big(ab,\, ay + xb + xy\big),$ 
which we denote by $A \ltimes B$.
Using the twisted actions defined by Lemma \ref{TA1}, we introduce another superalgebra  product on $A \oplus B$ given by 
 $$(a,x)\cdot_H (b,y) = \big(ab,\, l_H(a,y) + r_H(x,b) + xy\big),$$ 
and denote the resulting superalgebra by $A \ltimes_H B$.
\begin{thm}
Let $A, B$ be associative superalgebras and $B$ an associative $A$-bimodule. Let $H : A \to B$ be a linear map. The following are equivalent:
\begin{enumerate}
\item $H$ is a crossed homomorphism.
\item The map $i_H : A \to A \ltimes B$, $a \mapsto (a, H(a))$, is an superalgebra morphism.
\item The graph $\mathrm{Gr}(H) = \{(a,H(a)) \mid a \in A\}$ is a subsuperalgebra of $A \ltimes B$.
\item The operations $l_H, r_H$ define an $A$-bimodule structure on $B$ and the map
\[
\bar{H} : A \ltimes_H B \to A \ltimes B, \quad (a,x) \mapsto (a, H(a)+x)
\]
is a superalgebra morphism.
\end{enumerate}
\end{thm}

\begin{proof}
\textbf{$1 \Leftrightarrow 2$:}  
The map $i_H$ is an algebra morphism iff $  i_H(ab) = i_H(a)i_H(b),$
i.e.
\[
(ab, H(ab)) = (a,H(a))\cdot (b,H(b)) = (ab, aH(b) + H(a)b + H(a)H(b)).
\]
This holds iff $H$ is a crossed homomorphism.

\medskip
\textbf{ $1\Leftrightarrow 3$ :}  
$\mathrm{Gr}(H)$ is a subsuperalgebra of $A \ltimes B$ iff
$ (a,H(a))\cdot (b,H(b)) \in \mathrm{Gr}(H),$ 
i.e. iff $H(ab) = aH(b) + H(a)b + H(a)H(b).$

\medskip
\textbf{ $ 1 \Leftrightarrow 4$:}  
$\bar{H}$ is a superalgebra morphism iff
$\bar{H}((a,x)\cdot_H (b,y)) = \bar{H}(a,x)\cdot \bar{H}(b,y).$ 
Expanding both sides:
\[
\bar{H}(ab, ay + H(a)y + xb + xH(b) + xy)
= (ab, H(ab) + ay + H(a)y + xb + xH(b) + xy),
\]
and 
\[
(a,H(a)+x)\cdot (b,H(b)+y)
= (ab, a(H(b)+y) + (H(a)+x)b + (H(a)+x)(H(b)+y)).
\]
Equality holds iff $H$ satisfies the crossed homomorphism condition \ref{SCH1}.

\end{proof}

\medskip

Consider the extension of associative algebras
\[
0 \longrightarrow B \xrightarrow{i} A \oplus B \xrightarrow{p} A \longrightarrow 0.
\]
A section $s : A \to A \oplus B$ must be of the form $s(a) = (a,H(a))$.  
By the theorem above, $s$ is an algebra morphism iff $H$ is a crossed homomorphism.

\medskip

Let $ V=V_{0}\oplus V_{1}$ 
be a $\mathbb Z_2$-graded vector space (or super vector space) over a field $\mathbb{k}$. 
Write 
\[
T(V)=\bigoplus_{n\ge 0} V^{\otimes n},
\]
The $\mathbb Z_2$-grading   on $V$ is induces  $\mathbb Z_2$-grading on $T(V)$. For homogeneous elements  $ v_i\in V,$
the degree of $v_1\otimes \cdots \otimes v_n$  is defined by
\[ 
|v_1\otimes \cdots \otimes v_n|
=
|v_1|+\cdots +|v_n|
\pmod 2.
\]
Hence  $ T(V)=T(V)_{0}\oplus T(V)_{1},$ 
where $$  T(V)_{0}
=
\bigoplus_{\substack{n\ge0\\ |v_1|+\cdots+|v_n|=0}}
\Bbbk(v_1\otimes\cdots\otimes v_n),\qquad T(V)_{1}=\bigoplus_{\substack{n\ge 1\\ |v_1|+\cdots+|v_n|= 1}}
\Bbbk(v_1\otimes\cdots\otimes v_n).$$
The tensor concatenation  given by 
\[
(u_1\otimes \cdots \otimes u_m)
(v_1\otimes \cdots \otimes v_n)
=
u_1\otimes \cdots \otimes u_m
\otimes
v_1\otimes \cdots \otimes v_n.
\]
defines a multiplication on $T(V).$ With respect to this multiplication $T(V)$  is an associative superalgebra. This is called  tensor superalgebra of $V$.

Let  
$d : V \to V$ be a homogeneous linear map of degree $0 \in \mathbb{Z}_2$. 

We extend $d$ to a linear map 
\[
H_d : T(V) \to T(V)
\]
by defining, for homogeneous elements $v_1,\dots,v_n \in V$,
\[
\begin{aligned}
H_d(v_1 \otimes \cdots \otimes v_n)
&= \sum_{i}  v_1 \otimes \cdots \otimes d(v_i) \otimes \cdots \otimes v_n \\
&\quad + \sum_{i<j} 
\, v_1 \otimes \cdots \otimes d(v_i) \otimes \cdots \otimes d(v_j) \otimes \cdots \otimes v_n \\
&\quad + \cdots + d(v_1)\otimes \cdots \otimes d(v_n).
\end{aligned}
\]
\begin{prop}
    The linear map $H_d$ is a crossed homomorphism on the tensor superalgebra $T(V)$. 
In other words, $(T(V), H_d)$ is a Sup  AssCH pair.
\end{prop}

\section{Cohomology of Crossed Homomorphisms on Associative Superalgebras}

In this section, we consider a differential graded Lie algebra
(dgLa) whose Maurer--Cartan elements are precisely crossed
homomorphisms on associative superalgebras.
This characterization enables us to define the cohomology
of a crossed homomorphism.

Let $V=V_{0}\oplus V_{1}$ be a $\mathbb{Z}_2$-graded vector space and consider
the graded vector space  $\bigoplus_{n\ge 0}\Hom(V^{\otimes n},V)$. We define  $ f\in \Hom(A^{\otimes m},B)$ to have   degree $(m-1,\alpha_f),$ where $\alpha_f\in \mathbb{Z}_2$ if for homogeneous $x_1,\ldots, x_m$ in $A$, $f(x_1,\ldots,x_m)$ is a homogeneous element of $B$ and  $\deg f(x_1,\ldots,x_m)-\sum_i \deg (x_i)$ is congruent to $\alpha_f$ modulo $2.$
There is a degree $(0,0)$ graded Lie bracket, defined by \cite{MR4611401} 

\begin{equation}\label{RA1}
\llbracket f,g\rrbracket
=
f\bullet g
-
(-1)^{(m-1)(n-1)+\alpha_f\alpha_g}
g\bullet f,\end{equation}
for homogeneous maps  $
f\in \Hom(V^{\otimes m},V),$ 
$g\in \Hom(V^{\otimes n},V),$ 
where
\begin{align*}
(f\bullet g)(v_1,\ldots,v_{m+n-1})
&={}
\sum_{i=1}^{m}
(-1)^{(i-1)(n-1)}
(-1)^{\alpha_g(|v_1|+\cdots+|v_{i-1}|)}
\\
&\qquad
f(v_1,\ldots,v_{i-1},
g(v_i,\ldots,v_{i+n-1}),
\ldots,v_{m+n-1}).
\end{align*}
  With defferential $0$ $, (\bigoplus_{n\ge 0}\Hom(V^{\otimes n},V), \llbracket,\rrbracket,0) $ is a differential graded Lie algebra. 
 $\mu\in \Hom(V^{\otimes 2},V)$ is Maurer-Cartan element iff 
$\llbracket \mu,\mu\rrbracket=0$. Now  $\llbracket \mu,\mu\rrbracket=0$ if and only if $\mu\bullet\mu =0.$ 
$$\mu\bullet\mu(a,b,c)=\mu(\mu(a,b),c)+ (-1)^{1+0.|a|}\mu(a,\mu(b,c))=\mu(\mu(a,b),c)- \mu(a,\mu(b,c)),$$ $\forall$ homogeneous $a,b,c\in V$.
Thus $\mu$ is a Maurer-Cartan element of the differential graded Lie algebra $ (\bigoplus_{n\ge 0}\Hom(V^{\otimes n},V), \llbracket,\rrbracket,0) $ if and only if $\mu$ defines an associative superalgebra product on $V.$ 

Let $A$ and $B$ be two associative superalgebras and suppose that
$B$ is an associative $A$-bimodule superalgebra.
Denote the associative products on $A$ and $B$
respectively by $\mu_A$ and $\mu_B$.
Let 
$l:A\otimes B\to B,
\quad
r:B\otimes A\to B$                                                
be the left and right action maps. Consider the $\mathbb{Z}_2$-graded  vector space  $
V=A\oplus B$
and the graded Lie algebra $$ 
\bigoplus_{n\ge0}\Hom((A\oplus B)^{\otimes n},A\oplus B)$$
equipped with bracket given by \ref{RA1}.
Then $ \bigoplus_{n\ge0}\Hom(A^{\otimes n},B)$ 
is an abelian graded Lie subalgebra. Observe that $\mu_B$ is a Maurer--Cartan element in the graded Lie algebra
$$
\bigoplus_{n\ge0}
\Hom((A\oplus B)^{\otimes n},A\oplus B).
$$
Hence it defines a differential $d_{\mu_B}$ of degree $(1,0)$ given by  $d_{\mu_B}:=
\llbracket \mu_B,-\rrbracket.$
Next we define a  derived bracket $\llbracket -,-\rrbracket_1$ on $ 
\bigoplus_{n\ge0}\Hom(A^{\otimes n},B)$  by $$ 
\llbracket f,g\rrbracket_1
=
(-1)^{m-1}
\llbracket
\llbracket \mu_B,f\rrbracket,
g
\rrbracket,
.$$
for homogeneous $ f\in \Hom(A^{\otimes m},B)$ and  $g\in \Hom(A^{\otimes n},B).$ Here, we define  $ f\in \Hom(A^{\otimes m},B)$ to have   degree $(m,\alpha_f),$ where $\alpha_f$ belongs to  $\mathbb{Z}_2$, if for any homogeneous $x_1,\ldots, x_m$ in $A$, $f(x_1,\ldots,x_m)$ is a homogeneous element of $B$ and $\deg f(x_1,\ldots,x_m)-\sum_i \deg (x_i)$ is congruent to $\alpha_f$ modulo $2.$  By direct computation, We have $\llbracket f,g\rrbracket_1=-(-1)^{mn}\llbracket g,f\rrbracket_1$.
Explicitly, for homogeneous
$a_1,\ldots,a_{m+n}\in A$,
\begin{align*}
&\llbracket f,g\rrbracket_1
(a_1,\ldots,a_{m+n})\\
=& (-1)^{m-1}(\mu_B\bullet f)\bullet g(a_1,\ldots,a_{m+n})\\
=& (-1)^{m-1}(\mu_B\bullet f) (g(a_1,\ldots, a_n),a_{n+1},\ldots,a_{m+n})\\
&+(-1)^{m-1+m(n-1)+\alpha_g(|a_1+\ldots,|a_m|)}(\mu_B\bullet f)(a_1,\ldots, a_m,g(a_{m+1},\ldots, a_{m+n}))\\
=& (-1)^{\alpha_f \alpha_g+\alpha_f((|a_1|+\ldots,|a_n|)}\mu_B (g(a_1,\ldots, a_n),f(a_{n+1},\ldots,a_{m+n}))\\
&+(-1)^{mn+1+\alpha_g(|a_1|+\ldots,|a_m|)}\mu_B(f(a_1,\ldots, a_m),g(a_{m+1},\ldots, a_{m+n})).
\end{align*}

Thus $ \left(
\bigoplus_{n\ge0}\Hom(A^{\otimes n},B),
\llbracket-,-\rrbracket_1
\right)$  is a graded Lie algebra.  Also,  $ \llbracket f,g\rrbracket_1$ has degree $(m+n,\alpha_f+\alpha_g)$.

Moreover, it can be easily verified that $\mu_A+l+r$ is also a Maurer--Cartan element in dgLa $ (\bigoplus_{n\ge0}
\Hom((A\oplus B)^{\otimes n},A\oplus B), \llbracket -,-\rrbracket,0).$ 
Hence it induces a differential $  d_{\mu_A+l+r}=\llbracket\mu_A+l+r,-\rrbracket.$  For any homogeneous $a_1,\ldots, a_{m+1}$ belonging to $ A$ and $f\in \Hom(A^{\otimes m},B)$, we have 
\begin{align*}
&(d_{\mu_A+l+r}f)(a_1,\ldots,a_{m+1})
\\
=& \llbracket \mu_A+l+r,f \rrbracket(a_1,\ldots,a_{m+1})\\
={}&
(-1)^{m-1+\alpha_f|a_1|}
a_1f(a_2,\ldots,a_{n+1})+f(a_1,\ldots,a_n)a_{n+1}
\\
&\quad
-
(-1)^{m-1}
\sum_{i=1}^{n}
(-1)^{i-1}
f(a_1,\ldots,a_{i-1},
a_ia_{i+1},
\ldots,a_{n+1}).
\end{align*}
This implies that $d=d_{\mu_A+l+r}$ is a differential of degree $(1,0)$ on the graded space $ 
\bigoplus_{n\ge0}\Hom(A^{\otimes n},B)$.
Furthermore, we have  $\llbracket \mu_A+l+r,\mu_B \rrbracket=0.$  Therefore the differential $d$ is a derivation of degree $(1,0)$ with respect to
$\llbracket-,-\rrbracket_1$. Thus we conclude that  $ \left(\bigoplus_{n\ge0}\Hom(A^{\otimes n},B),
\llbracket-,-\rrbracket_1,d\right)$  is a differential graded Lie algebra.

\begin{prop}
Let $A$ and $B$ be two associative superalgebras. Suppose that $A$ acts on $B$ from left and right sides. A linear map $  H:A\to B$ 
of degree $0$ is a crossed homomorphism from $A$ to $B$
if and only if $  H\in Hom(A,B)$ 
is a Maurer--Cartan element in the differential graded Lie algebra $$ 
\left(
\bigoplus_{n\ge0}\Hom(A^{\otimes n},B),
\llbracket-,-\rrbracket_1,
d
\right).$$ 
\end{prop}
\begin{proof}
For a linear map $H:A\to B,$ in the dgla $\left(
\bigoplus_{n\ge0}\Hom(A^{\otimes n},B),
\llbracket-,-\rrbracket_1,
d
\right)$, we have
$\left(
dH+\frac{1}{2}\llbracket H,H\rrbracket_1
\right)(a,b)
=
aH(b)+H(a)b-H(ab)+H(a)H(b),$ 
for all homogeneous $a,b\in A$.
Hence $dH+\frac12\llbracket H,H\rrbracket_1=0$  if and only if $$
H(ab)=
aH(b)+H(a)b+H(a)H(b).$$
\end{proof}
A crossed homomorphism $  H:A\to B$ being a Maurer--Cartan element  induces a differential $d_H=d+\llbracket H,-\rrbracket_1$ 
on the graded Lie algebra $$ \left(
\bigoplus_{n\ge0}\Hom(A^{\otimes n},B),
\llbracket-,-\rrbracket_1
\right).$$
Define $C^n(A,B)=\Hom(A^{\otimes n},B)$ and $C^\ast(A,B)=\bigoplus_{n\ge0}C^n(A,B).$ 
Then a crossed homomorphism determines a dgLa $ (C^\ast(A,B), \llbracket-,- \rrbracket_1,d_H).$
\begin{thm}
Let $A,B$ be associative superalgebras and suppose that $B$ is an associative $A$-bimodule superalgebra. Assume that $ H:A\to B$ is a crossed homomorphism. Then for any linear map  $H':A\to B,$ 
the sum $H+H'$ is again a crossed homomorphism if and only if
$H'$ is a Maurer-Cartan element in the dgLa $ (C^\ast(A,B),
\llbracket-,-\rrbracket_1,
d_H).$ 
\end{thm}
\begin{proof}
We compute \begin{align*}
&d(H+H')
+
\frac{1}{2}
\llbracket H+H',H+H'\rrbracket_1\\
=&{}
d(H+H')
+
\frac12
\Bigl(
\llbracket H,H\rrbracket_1
+
2\llbracket H,H'\rrbracket_1
+
\llbracket H',H'\rrbracket_1
\Bigr)
\\
={}&
d_H(H')
+
\frac12
\llbracket H',H'\rrbracket_1,
\end{align*}
since $H$ itself satisfies the Maurer--Cartan equation.
Hence the result follows.
\end{proof}
For any integer $k\ge0$, define $Z_H^k(A,B)
=
\{f\in C^k(A,B)\mid d_H(f)=0\}$ 
and $ 
B_H^k(A,B)
=
\{d_Hf\mid f\in C^{k-1}(A,B)\}.$ 
Then $ B_H^k(A,B)\subseteq Z_H^k(A,B)$ is a subspace. For $k\ge 0$, the quotient space
$H_H^k(A,B)
=
\frac{Z_H^k(A,B)}{B_H^k(A,B)}$ 
are called the $k$th \emph{cohomology group of the crossed homomorphism} $H$.

\section{Formal Deformations of Crossed Homomorphisms}

Deformations of algebraic structures and morphisms play an important role
in deformation theory.
In this section, we study deformations of crossed homomorphisms
between associative superalgebras.

\subsection{Formal Deformations of Crossed Homomorphisms}

Let $A=A_{0}\oplus A_{1}$ be an associative superalgebra over a field
$\mathbb{K}$ of characteristic zero, and let $M$ be an associative
$A$-superalgebra endowed with compatible left and right $A$-actions.
Let $H:A\longrightarrow B$ be  a crossed homomorphism. Let $A[[t]]$ and $B[[t]]$  be formal power series on $A$ and $B$, respectively and  $H_t=\sum_{i=0}^{\infty}t^iH_i$
where $H_0=H$ and each $H_i:A\rightarrow B$ is an even linear map. Note that $A[[t]]$ and $B[[t]]$ are associative superalgebras and  $B[[t]]$ is an associative $A[[t]]$-bimodule.
A \emph{formal one parameter deformation} of $H$ is a linear map $H_t:A[[t]]\longrightarrow B[[t]]$  such that 
\begin{equation}\label{DefCH}
H_t(xy)=H_t(x)\cdot y+x\cdot H_t(y)+H_t(x)H_t(y),
\end{equation}
for all homogeneous $x,y\in A$.
Substituting $  H_t=\sum \limits_{i=0}^{\infty}t^iH_i$
into Equation \ref{DefCH}, we get    
\begin{equation}\label{nth}
H_n(xy)
=
H_n(x)\cdot y
+
x\cdot H_n(y)
+
\sum_{i+j=n}H_i(x)H_j(y),
\qquad n\ge0.
\end{equation}
The equation corresponding to $n=0$ is precisely the crossed homomorphism
identity for $H$. For $n=1$, we obtain
\begin{equation}\label{inf}
H_1(xy)
=
H_1(x)\cdot y
+
x\cdot H_1(y)
+
H(x)H_1(y)
+
H_1(x)H(y),
\end{equation}
From Equation \ref{inf}, we $1$-conclude that $H_1$ is a cocycle in $C^*(A,B).$  $H_1$ is called \emph{infinitesimal } of the deformation $H_t$.
We can write Equation \ref{nth} as 
\begin{equation}\label{nth1}
-\sum_{\substack{i+j=n\\i,j\ge 1}}H_i(x)H_j(y),
=
-H_n(xy)+H_n(x)\cdot y
+
x\cdot H_n(y)
+ H_n(x)H(y)+H(x)H_n(y),
\end{equation}
for all $n\ge 0.$
By direct calculation we find that $$d_H H_n(x,y)=-H_n(xy)+H_n(x)\cdot y + x\cdot H_n(y) + H_n(x)H(y)+H(x)H_n(y)$$ and $$-\frac{1}{2} \sum_{\substack{i+j=n\\i,j\ge 1}}\llbracket H_i,H_j \llbracket_1(x,y)=-\sum_{\substack{i+j=n\\i,j\ge 1}}H_i(x)H_j(y).$$ Hence Equation \ref{nth1} can be written as 
\begin{equation}\label{obs}
  -\frac{1}{2} \sum_{\substack{i+j=n\\i,j\ge 1}}\llbracket H_i,H_j\rrbracket_1(x,y)= d_H (H_n) 
\end{equation}
 A formal  one parameter deformation of order $n$ is a linear map $H_t: \frac{A[[t]]}{(t^{n+1})}\to \frac{B[[t]]}{(t^{n+1})}$  such that  
\begin{itemize}
    \item[(a)] $H_t=\sum_{i=0}^{n}t^iH_i $, where $H_0=H$ and $H_i:A\to B$ is an even linear map for $1\le i\le n.$
    \item[(b)] $H_t(xy)=H_t(x)\cdot y+x\cdot H_t(y)+H_t(x)H_t(y)$.
\end{itemize}
By direct computation, for all $1\le m\le n$, we get following Equation 
\begin{equation}\label{obs1}
    -\frac{1}{2} \sum_{\substack{i+j=m\\i,j\ge 1}}\llbracket H_i,H_j\rrbracket_1(x,y)= d_H (H_m).
\end{equation}
 We denote the left hand side of Equation \ref{obs1} by $ob_H^m$ and call it \textit{Obstruction} of order $m.$
\begin{thm}
Obstructions are cocycles, that is $d_H(ob_H^m)=0$.
\end{thm}
\begin{proof} By using Equation \ref{obs1}, we have 
\begin{align*}
        &d_H( -\frac{1}{2} \sum_{\substack{i+j=m\\i,j\ge 1}}\llbracket H_i,H_j\rrbracket_1)\\
        =&-\frac{1}{2}\sum_{\substack{i+j=m\\i,j\ge 1}}\{\llbracket d_HH_i,H_j \rrbracket_1-\llbracket H_i,d_HH_j\rrbracket_1\}\\
        =&\frac{1}{4}\sum_{\substack{i+j=m\\i,j\ge 1}}\llbracket \sum_{\substack{i_1+i_2=i\\i_1,i_2\ge 1}}\llbracket H_{i_1},H_{i_2}\rrbracket_1,H_j\rrbracket_1
        -\frac{1}{4}\sum_{\substack{i+j=m\\i,j\ge 1}}\llbracket H_i, \sum_{\substack{j_1+j_2=j\\j_1,j_2\ge 1}}\llbracket H_{j_1},H_{j_2}\rrbracket_1\rrbracket_1\\
        =&\frac{1}{4}\sum_{\substack{i_1+i_2+j=m\\i_1,i_2,j\ge 1}}\llbracket \llbracket H_{i_1},H_{i_2}\rrbracket_1,H_j\rrbracket_1 +\frac{1}{4}\sum_{\substack{i+j_1+j_2=m\\i,j_1,j_2\ge 1}}\llbracket \llbracket H_{j_1},H_{j_2}\rrbracket_1,H_i\rrbracket_1\\
        =&\frac{1}{2}\sum_{\substack{i+j+k=m\\i,j,k\ge 1}}\llbracket \llbracket H_{i},H_{j}\rrbracket_1,H_k\rrbracket_1\\
        =&0 \qquad \text{(By using graded Jacobi Identity)}
   \end{align*}
\end{proof}
As a consequence of previous theorem, we have following result.
\begin{cor}
 If $H_H^2(A,B)=0,$ then cohomology class of $Ob_H^m$ is $0$, for each $1\le m\le n$. Hence, every order $n$ deformation can be extended to an order $n+1$ deformation. 
\end{cor}

%\section*{}
%\bibliographystyle{alpha}
 \bibliography{raj.bib}

 \end{document}